\documentclass[11pt,final]{amsart}

\usepackage[
  a4paper,
  margin=2.5cm
]{geometry}
\usepackage[T1]{fontenc}
\usepackage[utf8]{inputenc}
\usepackage{lmodern}
\usepackage{microtype}
\usepackage{mathtools}
\usepackage{amssymb}
\usepackage{mathrsfs}
\usepackage{bm}
\usepackage{enumitem}
\usepackage{booktabs}
\usepackage{xcolor}
\usepackage{hyperref}
\usepackage{aliascnt}
\usepackage[nameinlink,capitalize,noabbrev]{cleveref}

\hypersetup{
colorlinks=true,
linkcolor=black
}

\numberwithin{equation}{section}

\newcommand{\Sph}{\mathbb{S}}
\newcommand{\dd}{\,\mathrm{d}}

\newcommand{\cH}{\mathcal{H}}
\newcommand{\cD}{\mathcal{D}}

\newcommand{\cU}{\mathcal{U}}
\newcommand{\ip}[2]{\left\langle #1,#2\right\rangle}
\DeclareMathOperator{\diver}{div}
\DeclareMathOperator{\spanop}{span}
\DeclareMathOperator{\spec}{spec}

\newtheorem{theorem}{Theorem}[section]
\newaliascnt{proposition}{theorem}
\newtheorem{proposition}[proposition]{Proposition}
\aliascntresetthe{proposition}
\newaliascnt{lemma}{theorem}
\newtheorem{lemma}[lemma]{Lemma}
\aliascntresetthe{lemma}
\newaliascnt{corollary}{theorem}
\newtheorem{corollary}[corollary]{Corollary}
\aliascntresetthe{corollary}
\theoremstyle{definition}
\newaliascnt{definition}{theorem}

\aliascntresetthe{definition}
\theoremstyle{remark}
\newaliascnt{remark}{theorem}
\newtheorem{remark}[remark]{Remark}
\aliascntresetthe{remark}

\title[Complete spectrum at a Sobolev extremal]
{The complete spectrum of the linearized \(p\)-Laplacian\\
at a Sobolev extremal}
\author{Yitian Zhang}
\date{}

\begin{document}

\begin{abstract}
Let \(1<p<n\), and let
\[
 v(x)=\left(1+|x|^{p/(p-1)}\right)^{-(n-p)/p}
\]
be the standard radial extremal for the sharp Sobolev inequality. We determine all eigenvalues and eigenspaces of the linearized \(p\)-Laplacian at \(v\), understood as the self-adjoint operator defined by its closed quadratic form in \(L^2(\mathbb{R}^n,v^{p^*-2}\,\mathrm{d}x)\).  After decomposition into spherical harmonics, we use an explicit gauge transformation and a change of variables to identify each radial block, at the level of closed quadratic forms, with a shifted Jacobi operator.  This yields a complete eigenbasis indexed by \((\ell,k)\in\mathbb{N}_0^2\), where \(\ell\) is the angular degree and \(k\) is the radial mode number.
\end{abstract}

\maketitle

\section{Introduction}

For \(1<p<n\), let
\[
q=p^*\coloneq \frac{np}{n-p}.
\]
The sharp Sobolev inequality on \(\mathbb{R}^n\) states that
\begin{equation}\label{eq:sharp-sobolev}
  S_{n,p}
  \left(
      \int_{\mathbb{R}^n}|u|^q\,\dd x
  \right)^{p/q}
  \leq
  \int_{\mathbb{R}^n}|\nabla u|^p\,\dd x,
  \quad
  u\in \dot W^{1,p}(\mathbb{R}^n),
\end{equation}
where \(S_{n,p}>0\) is the optimal constant.  The equality cases were classified by Aubin and Talenti \cite{aubin_problemes_1976,Talenti1976}: up to translations, dilations, and multiplication by a constant, they are given by the radial Aubin--Talenti bubble
\begin{equation}\label{eq:v-definition}
  v(x)=v(r)\coloneq (1+r^a)^{-(n-p)/p},
  \quad
  r=|x|,
  \quad
  a\coloneq \frac{p}{p-1}.
\end{equation}

Indeed, minimizing the \(p\)-Dirichlet energy under a fixed \(L^q\)-norm constraint, gives the Euler--Lagrange equation
\[
-\Delta_p u
\coloneq 
-\diver\left(|\nabla u|^{p-2}\nabla u\right)
=
\kappa |u|^{q-2}u
\]
for a suitable Lagrange multiplier \(\kappa>0\).  With the normalization in \eqref{eq:v-definition}, the bubble \(v\) satisfies
\begin{equation}\label{eq:bubble-equation}
  -\Delta_p v
  =
  \kappa_{n,p}v^{q-1},
  \quad
  \kappa_{n,p}
  \coloneq 
  n\left(\frac{n-p}{p-1}\right)^{p-1}.
\end{equation}

To pass from the nonlinear Euler--Lagrange equation to its linearized theory, consider the corresponding Lagrangian
\[
\mathcal{L}(u)
\coloneq 
\frac{1}{p}\int_{\mathbb{R}^n}|\nabla u|^p\,\dd x
-
\frac{\kappa_{n,p}}{q}
\int_{\mathbb{R}^n}|u|^q\,\dd x.
\]
The second variation of its Dirichlet part at \(v\) is determined by the symmetric matrix
\begin{equation}\label{eq:Av}
  A_v(x)
  \coloneq 
  |\nabla v|^{p-2}
  \left(
    I+(p-2)\frac{\nabla v}{|\nabla v|}
    \otimes
    \frac{\nabla v}{|\nabla v|}
  \right).
\end{equation}
Since \(v\) is radial, the radial eigenvalue of \(A_v\) is \((p-1)|v'|^{p-2}\), whereas each tangential eigenvalue is \(|v'|^{p-2}\).  Consequently, the full second variation is
\[
\mathcal{L}''(v)[\phi,\psi]
=
\int_{\mathbb{R}^n} A_v\nabla\phi\cdot\nabla\psi\,\dd x
-
\kappa_{n,p}(q-1)
\int_{\mathbb{R}^n} v^{q-2}\phi\psi\,\dd x.
\]
This quadratic form relative to the weighted space \(L^2(\mathbb{R}^n,v^{q-2}\dd x)\) leads to the generalized eigenvalue problem
\begin{equation}\label{eq:global-eigenproblem-intro}
  -\diver(A_v\nabla\phi)=\mu v^{q-2}\phi
  \quad\text{in }\mathbb{R}^n.
\end{equation}
The operator in \eqref{eq:global-eigenproblem-intro} is understood through the closure of its energy form in \(L^2(\mathbb{R}^n,v^{q-2}\dd x)\), in accordance with the representation theory of closed quadratic forms
\cite{kato_perturbation_2012}.  Variational spectral theory for linearized \(p\)-Laplace equations was developed by Castorina, Esposito, and Sciunzi \cite{castorina_spectral_2011}.  In the present whole-space problem, the formal differential expression \(-\diver(A_v\nabla\cdot)\) does not by itself specify a self-adjoint
operator.  We therefore define \(\mathcal{L}_v\) through the closure of its quadratic form; this construction determines both the operator domain and the endpoint conditions
inherited from the original energy space. Nevertheless, the radial and tangential eigenvalues of \(A_v\) are uniformly comparable to \(|\nabla v|^{p-2}\).  The compact embedding proved by Figalli and Zhang \cite{figalli_sharp_2022} therefore implies that the closed form domain embeds compactly into \(L^2(\mathbb{R}^n,v^{q-2}\dd x)\).  Consequently, the associated operator has compact resolvent: its spectrum is purely discrete, and it remains only to determine its point spectrum and the corresponding eigenspaces. 

Problem \eqref{eq:global-eigenproblem-intro} is motivated by quantitative Sobolev stability. For \(p=2\), Bianchi and Egnell \cite{bianchi_note_1991} proved that the Sobolev deficit controls the squared homogeneous \(H^1\)-distance to the manifold of optimizers. Figalli and Neumayer \cite{figalli_gradient_2018} established gradient stability in the range \(p\geq 2\). At the linearized level, Pistoia and Vaira \cite{pistoia_vaira_2021} proved the nondegeneracy of the critical \(p\)-Laplace bubble for every \(1<p<n\), showing that the kernel of the linearized operator is generated precisely by the infinitesimal dilation and translation modes. Figalli and Zhang \cite{figalli_sharp_2022} subsequently proved sharp gradient stability for every \(1<p<n\), with the natural distance and the optimal exponent. K\"onig \cite{konig_sharp_2023} proved that the optimal global Bianchi--Egnell constant is strictly smaller than the constant predicted by the local spectral gap. Thus, the local spectrum alone does not determine the optimal stability constant.

When \(p=2\), conformal covariance and stereographic projection reduce \eqref{eq:global-eigenproblem-intro} to a spherical problem.  Its spectrum is therefore explicit and the corresponding eigenspaces are identified with the spherical harmonics of degree \(N\) on \(\Sph^n\).  For \(p\ne2\), the radial and tangential coefficients of \(A_v\) differ, and this direct conformal reduction is no longer available.

The use of hypergeometric equations in explicit spectral calculations has several relevant precedents.  Denzler and McCann \cite{denzler_fast_2005} obtained a complete hypergeometric spectrum for a linearized fast-diffusion operator.  Wei and Wu \cite{wei_stability_2024} derived a two-index spectrum in a Caffarelli--Kohn--Nirenberg problem.  For a class of weighted \(p\)-Laplace equations, Deng and Tian \cite{deng_caffarelli-kohn-nirenberg-type_2025} classified all solutions of the linearized Euler--Lagrange equation at the radial extremal and determined the first two distinct spectral levels of the corresponding weighted operator, including their multiplicities.  For the critical quasilinear H\'enon equation, Dai, Duan, Gui, and Li \cite{dai_nonradial_2026} explicitly computed the first radial eigenpair in every spherical-harmonic sector. They also classified the eigenspace corresponding to the eigenvalue \(\mu=1\) and used qualitative information on the higher radial
eigenvalues to determine the Morse index. 

We determine the eigenvalues and eigenfunctions of \eqref{eq:global-eigenproblem-intro} for every spherical-harmonic degree and every radial mode.  After spherical-harmonic decomposition, a gauge transformation reduces each radial equation formally to a Jacobi equation.  We then establish the reduction at the level of closed quadratic forms, which fixes the Friedrichs endpoint conditions and allows Jacobi completeness to yield the full spectrum.

\subsection*{Notation for the spectral data}

Throughout the paper we use
\begin{equation}\label{eq:basic-parameters}
  P\coloneq p-1,\quad
  B\coloneq \frac{n-p}{P},\quad
  \delta\coloneq 2+\frac{n(p-2)}{p},\quad
  C_{n,p}\coloneq \frac{p^2}{P}B^{p-2}.
\end{equation}
Let
\[
\lambda_\ell\coloneq \ell(\ell+n-2),\quad \ell\in\mathbb{N}_0,
\]
be the eigenvalues of \(-\Delta_{\Sph^{n-1}}\), and define
\begin{align}
  D&\coloneq (n-1)P-1=nP-p,\label{eq:D}\\
  s_0&\coloneq 0,\label{eq:s0}\\
  s_\ell
  &\coloneq \frac{-D+\sqrt{D^2+4P\lambda_\ell}}{2P},
  &&\ell\geq 1,\label{eq:sell}\\
  \tau_\ell
  &\coloneq \frac{n-p+\sqrt{(n-p)^2+4P\lambda_\ell}}{2P},
  &&\ell\geq 0,\label{eq:tauell}\\
  m_\ell&\coloneq \frac{P}{p}(s_\ell+\tau_\ell).\label{eq:mell}
\end{align}

For later use, set
\begin{equation}\label{eq:jacobi-parameters-intro}
  \alpha_\ell\coloneq \frac{P(2s_\ell+n)}{p}-1,
  \quad
  \beta_\ell\coloneq \delta+2m_\ell-\frac{P(2s_\ell+n)}{p}-1.
\end{equation}
We denote by \(P_k^{(\alpha,\beta)}\) the Jacobi polynomial in the standard normalization of \cite{szego_orthogonal_1975}.

Let \(\cH_\ell\) be the space of spherical harmonics of degree \(\ell\) on \(\Sph^{n-1}\), and write
\begin{equation}\label{eq:dnl}
  d_{n,\ell}\coloneq \dim\cH_\ell
  =
  \binom{n+\ell-1}{\ell}
  -
  \binom{n+\ell-3}{\ell-2},
\end{equation}
where the second binomial coefficient is understood to be zero when \(\ell<2\).

\begin{theorem}
\label{thm:main-spectrum}
Let \(1<p<n\), and let \(\mathcal{L}_v\) be the nonnegative self-adjoint
operator in
\[
\mathscr H_v=L^2(\mathbb R^n,v^{q-2}\dd x)
\]
defined by the closure of the quadratic form
\[
\phi\longmapsto
\int_{\mathbb R^n}A_v\nabla\phi\cdot\nabla\phi\,\dd x
\]
on smooth compactly supported functions that are constant near the
origin.

For each \(\ell\geq0\), choose an orthonormal basis
\(\{Y_{\ell,j}\}_{j=1}^{d_{n,\ell}}\) of \(\cH_\ell\), and define
\begin{equation}
\label{eq:full-eigenfunctions-intro}
  \Phi_{\ell,k,j}(r,\theta)
  =
  r^{s_\ell}(1+r^a)^{-m_\ell}
  P_k^{(\alpha_\ell,\beta_\ell)}
  \left(\frac{1-r^a}{1+r^a}\right)
  Y_{\ell,j}(\theta).
\end{equation}
Then \(\Phi_{\ell,k,j}\in\cD(\mathcal{L}_v)\) and
\[
\mathcal{L}_v\Phi_{\ell,k,j}
=
\mu_{\ell,k}\Phi_{\ell,k,j},
\qquad
\mu_{\ell,k}
=
C_{n,p}(m_\ell+k)(m_\ell+k+\delta-1).
\]
Moreover, the functions \(\Phi_{\ell,k,j}\), with
\[
\ell,k\in\mathbb N_0,
\qquad
1\leq j\leq d_{n,\ell},
\]
form a complete orthogonal basis of \(\mathscr H_v\).  Hence
\begin{equation}
\label{eq:complete-spectrum-intro}
  \spec(\mathcal{L}_v)
  =
  \{\mu_{\ell,k}:\ell,k\in\mathbb N_0\},
\end{equation}
Moreover
\begin{equation}\label{eq:crossing-rule-intro}
  \mu_{\ell,k}=\mu_{\ell',k'}
  \iff
  m_\ell+k=m_{\ell'}+k'.
\end{equation}
\end{theorem}

\begin{remark}
Set
\begin{equation}
\label{eq:separated-eigenspace}
  E_{\ell,k}
  \coloneq
  \spanop\{\Phi_{\ell,k,j}:1\leq j\leq d_{n,\ell}\}.
\end{equation}
Then for every \(\mu\in\spec(\mathcal{L}_v)\),
\begin{equation}
\label{eq:numerical-eigenspace-intro}
  \ker(\mathcal{L}_v-\mu I)
  =
  \bigoplus_{\substack{\ell,k\geq0\\ \mu_{\ell,k}=\mu}}
  E_{\ell,k},
\end{equation}
and
\[
\dim\ker(\mathcal{L}_v-\mu I)
=
\sum_{\substack{\ell,k\geq0\\ \mu_{\ell,k}=\mu}}
d_{n,\ell}<\infty.
\]  
\end{remark}

The first two distinct eigenspaces have the expected variational meaning.  The pair \((0,0)\) is generated by \(v\).  Moreover,
\[
m_0=\frac{n-p}{p},
\quad
m_1=\frac{n}{p}=m_0+1,
\]
so the pairs \((0,1)\) and \((1,0)\) produce the dilation and translation modes at the common eigenvalue.

Our application is the exact identification of the next distinct eigenspace.
\begin{corollary}
\label{cor:E3-intro}
Assume that \(n\geq 3\) and \(1<p<2\).  Write the distinct
eigenvalues of \(\mathcal{L}_v\) in increasing order as
\[
\mu_1<\mu_2<\mu_3<\cdots,
\]
and let
\[
E_j\coloneq\ker(\mathcal{L}_v-\mu_j I).
\]
Then the third distinct eigenvalue is
\begin{equation}
\label{eq:mu3-intro}
  \mu_3
  =
  \mu_{2,0}
  =
  C_{n,p}m_2(m_2+\delta-1),
\end{equation}
and its eigenspace is
\begin{equation}
\label{eq:E3-intro}
  E_3
  =
  E_{2,0}
  =
  \left\{
  r^{s_2}(1+r^a)^{-m_2}Y(\theta):
  Y\in\cH_2
  \right\}.
\end{equation}
Consequently,
\begin{equation}
\label{eq:E3-dimension-intro}
  \dim E_3
  =
  d_{n,2}
  =
  \frac{(n-1)(n+2)}{2}.
\end{equation}
\end{corollary}

\subsection*{Structure of the proof}

\Cref{sec:operator-framework} defines the self-adjoint operator associated with the closed quadratic form and uses spherical-harmonic decomposition to reduce the full problem to the radial operators \(T_\ell\).  \Cref{sec:jacobi-reduction} determines the admissible asymptotic powers at \(r=0\) and \(r=\infty\), and formally transforms each radial eigenvalue equation into a shifted Jacobi equation. \Cref{sec:form-identification} identifies the transformed closed form and proves that the resulting Jacobi functions satisfy the endpoint conditions and belong to the domains of the corresponding operators. \Cref{sec:complete-spectrum} uses the orthogonality and completeness of the Jacobi polynomials to prove that these functions form a complete eigenbasis for each \(T_\ell\).  Finally, \cref{sec:E3} combines the spectra of all angular sectors, characterizes all coincidences among the eigenvalues, proves the main theorem, and identifies the third distinct eigenspace when \(1<p<2\).

\section{The operator framework and angular reduction}
\label{sec:operator-framework}

\subsection{The closed form}

Differentiating \eqref{eq:v-definition} gives
\begin{equation}\label{eq:vprime}
  v'(r)=-B r^{1/P}(1+r^a)^{-n/p}.
\end{equation}
Hence, with
\begin{align}
  \rho(r)&\coloneq v(r)^{q-2}=(1+r^a)^{-\delta},
  \label{eq:rho-weight}\\
  w(r)&\coloneq |v'(r)|^{p-2}=B^{p-2}r^{(p-2)/P}(1+r^a)^{-n(p-2)/p},
  \label{eq:w-weight}
\end{align}
the matrix in \eqref{eq:Av} can be written as
\begin{equation}\label{eq:Av-radial}
 A_v=w(r)\left(I+(p-2)e_r\otimes e_r\right).
\end{equation}

Let
\begin{equation}\label{eq:global-core}
 \cD_{c,0}^{\infty}
 \coloneq 
 \left\{
 \phi\in C_c^\infty(\mathbb{R}^n):
 \phi\text{ is constant in a neighbourhood of }0
 \right\}.
\end{equation}
This is dense in \(\mathscr{H}_v=L^2(\mathbb{R}^n,\rho\dd x)\).  On this core set
\begin{equation}\label{eq:global-form}
 \mathfrak{a}_v[\phi,\psi]
 \coloneq 
 \int_{\mathbb{R}^n}
 A_v\nabla\phi\cdot\nabla\psi\,\dd x.
\end{equation}
It is a nonnegative densely defined closable form.  We use the same symbol for its closure and denote its form domain by \(\cD(\mathfrak{a}_v)\).  The representation theorem for closed forms yields a unique nonnegative self-adjoint operator \(\mathcal{L}_v\) \cite{kato_perturbation_2012} such that
\[
\mathfrak{a}_v[\phi,\psi]
=
\ip{\mathcal{L}_v\phi}{\psi}_{\mathscr{H}_v},
\quad
\phi\in\cD(\mathcal{L}_v),\quad \psi\in\cD(\mathfrak{a}_v).
\]
Thus \eqref{eq:global-eigenproblem-intro} is interpreted weakly, with the endpoint conditions determined by \(\cD(\mathfrak{a}_v)\).

\subsection{Spherical harmonics and radial Friedrichs operators}

Choose an orthonormal basis
\[
\{Y_{\ell,j}\}_{j=1}^{d_{n,\ell}}
\]
of \(\cH_\ell\), so that
\[
-\Delta_{\Sph^{n-1}}Y_{\ell,j}
=
\lambda_\ell Y_{\ell,j}.
\]
The spherical-harmonic transform is unitary from \(\mathscr{H}_v\) onto
\[
\widehat{\bigoplus}_{\ell=0}^{\infty}
\left(\mathscr{K}_v\otimes\cH_\ell\right),
\quad
\mathscr{K}_v
\coloneq 
L^2\left((0,\infty),\rho(r)r^{n-1}\dd r\right).
\]

For each \(\ell\), introduce the radial core
\begin{equation}\label{eq:radial-cores}
  \mathcal{D}_\ell^{\mathrm{rad}}
  \coloneq 
  \begin{cases}
    \left\{
    f\in C_c^\infty([0,\infty)):
    f\text{ is constant near }0
    \right\},
    & \ell=0,
    \\[2mm]
    C_c^\infty(0,\infty),
    & \ell\geq 1,
  \end{cases}
\end{equation}
and define on this core the nonnegative form
\begin{equation}\label{eq:radial-form}
 \mathfrak{b}_\ell[f,g]
 \coloneq 
 \int_0^\infty
 w(r)
 \left(
 P f'(r)g'(r)
 +
 \frac{\lambda_\ell}{r^2}
 f(r)g(r)
 \right)
 r^{n-1}\dd r.
\end{equation}
The form is closable in \(\mathscr{K}_v\).  We use the same notation \(\mathfrak{b}_\ell\) for its closure and denote its domain by
\[
\mathcal{D}_\ell(\mathfrak{b}_\ell):
=
\overline{\mathcal{D}_\ell^{\mathrm{rad}}}^{
\|\cdot\|_{\mathfrak{b}_\ell}},
\quad
\|f\|_{\mathfrak{b}_\ell}^2
\coloneq 
\|f\|_{\mathscr{K}_v}^2
+
\mathfrak{b}_\ell[f,f].
\]
Let \(T_\ell\) be the nonnegative self-adjoint operator in \(\mathscr{K}_v\) associated with the closed form
\(\left(\mathfrak{b}_\ell,\mathcal{D}_\ell(\mathfrak{b}_\ell)\right).\)

\begin{proposition}
\label{prop:angular-decomposition}
Under the spherical-harmonic transform,
\begin{equation}\label{eq:operator-direct-sum}
  \mathcal{L}_v
  \simeq
  \widehat{\bigoplus}_{\ell=0}^{\infty}
  \left(T_\ell\otimes I_{\cH_\ell}\right).
\end{equation}
\end{proposition}

\begin{proof}
Let\(\phi,\psi\in\mathcal{D}(\mathfrak{a}_v).\) Under the spherical-harmonic transform, write
\[
\phi(r,\theta)
=
\sum_{\ell=0}^{\infty}
\sum_{j=1}^{d_{n,\ell}}
f_{\ell,j}(r)Y_{\ell,j}(\theta),
\quad
\psi(r,\theta)
=
\sum_{\ell=0}^{\infty}
\sum_{j=1}^{d_{n,\ell}}
g_{\ell,j}(r)Y_{\ell,j}(\theta).
\]
Then
\[
f_{\ell,j},g_{\ell,j}
\in\mathcal{D}_\ell(\mathfrak{b}_\ell),
\quad
\forall\ell,j,
\]
and
\[
\sum_{\ell,j}
\|f_{\ell,j}\|_{\mathfrak{b}_\ell}^{2}
<\infty,
\quad
\sum_{\ell,j}
\|g_{\ell,j}\|_{\mathfrak{b}_\ell}^{2}
<\infty.
\]
Indeed, the corresponding assertions and the form identity hold on the core by the orthogonality of spherical harmonics.  A standard approximation argument in the form norm, using spherical-harmonic partial sums, extends them to the full form domain \(\mathcal{D}(\mathfrak{a}_v)\).  Consequently,
\begin{equation}\label{eq:form-direct-sum}
  \mathfrak{a}_v[\phi,\psi]
  =
  \sum_{\ell=0}^{\infty}
  \sum_{j=1}^{d_{n,\ell}}
  \mathfrak{b}_\ell
  [f_{\ell,j},g_{\ell,j}].
\end{equation}

Now let
\[
F(r,\theta)
=
\sum_{\ell=0}^{\infty}
\sum_{j=1}^{d_{n,\ell}}
h_{\ell,j}(r)Y_{\ell,j}(\theta)
\in\mathscr{H}_v.
\]
By the definitions of the operators associated with the closed forms,
\[
\begin{aligned}
  &\phi\in\mathcal{D}(\mathcal{L}_v),
  \quad
  \mathcal{L}_v\phi=F
  \\[1mm]
  &\iff
  \mathfrak{a}_v[\phi,\psi]
  =
  \langle F,\psi\rangle_{\mathscr{H}_v},
  \quad
  \forall\psi\in\mathcal{D}(\mathfrak{a}_v)
  \\[1mm]
  &\iff
  \sum_{\ell,j}
  \mathfrak{b}_\ell[f_{\ell,j},g_{\ell,j}]
  =
  \sum_{\ell,j}
  \langle h_{\ell,j},g_{\ell,j}\rangle_{\mathscr{K}_v},
  \quad
  \forall
  (g_{\ell,j})
  \in
  \widehat{\bigoplus}_{\ell,j}
  \mathcal{D}_\ell(\mathfrak{b}_\ell)
  \\[1mm]
  &\iff
  \mathfrak{b}_\ell[f_{\ell,j},g]
  =
  \langle h_{\ell,j},g\rangle_{\mathscr{K}_v},
  \quad
  \forall g\in\mathcal{D}_\ell(\mathfrak{b}_\ell),
  \quad
  \forall\ell,j
  \\[1mm]
  &\iff
  f_{\ell,j}\in\mathcal{D}(T_\ell),
  \quad
  T_\ell f_{\ell,j}=h_{\ell,j},
  \quad
  \forall\ell,j.
\end{aligned}
\]
Since \(F\in\mathscr{H}_v\), the family \((h_{\ell,j})=(T_\ell f_{\ell,j})\) is square summable.  This is exactly the domain and action of the operator direct sum in
\eqref{eq:operator-direct-sum}.
\end{proof}

By \cref{prop:angular-decomposition}, the eigenvalue problem for
\(\mathcal{L}_v\) is equivalent to the family of eigenvalue problems for
\(T_\ell\), \(\ell\geq 0\).  More precisely,
\[
\ker(\mathcal{L}_v-\mu)
\simeq
\widehat{\bigoplus}_{\ell=0}^{\infty}
\left(
\ker(T_\ell-\mu)\otimes\mathcal{H}_\ell
\right).
\]
It therefore remains to determine the eigenvalues and eigenfunctions
of \(T_\ell\) for each \(\ell\).

For a sufficiently regular radial function, the equation \(T_\ell f=\mu f\) has the differential expression
\begin{equation}\label{eq:radial-divergence-equation}
  -\frac{1}{\rho r^{n-1}}
  \frac{\dd}{\dd r}
  \left(Pwr^{n-1}f'\right)
  +\frac{w\lambda_\ell}{\rho r^2}f
  =\mu f.
\end{equation}
Substitution of \eqref{eq:rho-weight} and \eqref{eq:w-weight}, followed by division by the coefficient of \(f''\), gives
\begin{equation}\label{eq:radial-normalized-equation}
  \begin{aligned}
    P f''
    +
    \left[
    \frac{nP-1}{r}
    -n(p-2)\frac{r^{a-1}}{1+r^a}
    \right]f'
    -\frac{\lambda_\ell}{r^2}f
    +\widehat\mu\frac{r^{a-2}}{(1+r^a)^2}f
    =0.
  \end{aligned}
\end{equation}
Here and below,
\begin{equation}\label{eq:muhat}
  \widehat\mu\coloneq \mu B^{2-p}.
\end{equation}

\section{Reduction of every radial block to a Jacobi equation}
\label{sec:jacobi-reduction}

\subsection{Indicial roots at zero and infinity}

To determine the indicial roots at \(r=0\), observe that
\[
\frac{r^{a-1}}{1+r^a}
=O(r^{a-1})=o(r^{-1}),
\qquad
\frac{r^{a-2}}{(1+r^a)^2}
=O(r^{a-2})=o(r^{-2}).
\]
Hence the leading Euler equation associated with
\eqref{eq:radial-normalized-equation} is
\[
P f''
+\frac{nP-1}{r}f'
-\frac{\lambda_\ell}{r^2}f
=0.
\]
Substituting the trial function \(f(r)=r^s\) gives
\begin{equation}\label{eq:origin-indicial}
  Ps(s-1)+(nP-1)s-\lambda_\ell=0,
\end{equation}
or, equivalently,
\begin{equation}\label{eq:origin-indicial-simplified}
  Ps^2+Ds-\lambda_\ell=0.
\end{equation}
For \(\ell\geq1\), the two roots have opposite signs, and the positive root is \(s_\ell\) from \eqref{eq:sell}.  For \(\ell=0\), the radial form domain determined by \eqref{eq:radial-cores} selects \(s_0=0\).

At \(r=\infty\), one has
\[
\frac{r^{a-1}}{1+r^a}
=r^{-1}+o(r^{-1}),
\qquad
\frac{r^{a-2}}{(1+r^a)^2}
=O(r^{-a-2})=o(r^{-2}).
\]
Therefore the leading Euler equation is
\[
P f''
+\frac{n-1}{r}f'
-\frac{\lambda_\ell}{r^2}f
=0.
\]
Substituting the trial function \(f(r)=r^{-\tau}\) yields
\begin{equation}\label{eq:infinity-indicial}
  P\tau^2-(n-p)\tau-\lambda_\ell=0.
\end{equation}
The positive root is \(\tau_\ell\) from \eqref{eq:tauell}, and it corresponds to the decaying power \(r^{-\tau_\ell}\).

\begin{lemma}
\label{lem:parameter-bounds}
For every \(\ell\geq 0\),
\[
\alpha_\ell>-1,\quad \beta_\ell>0.
\]
For every \(\ell\geq 1\), one also has \(\alpha_\ell>0\).
\end{lemma}

\begin{proof}
By \eqref{eq:jacobi-parameters-intro},
\[
\alpha_\ell+1=\frac{P(2s_\ell+n)}{p}>0.
\]
If \(\ell\geq 1\), then \(s_\ell\geq s_1=1/P\), and hence
\[
\alpha_\ell+1
\ge\frac{Pn+2}{p}>1.
\]
Using \(m_\ell=P(s_\ell+\tau_\ell)/p\) in the definition of \(\beta_\ell\) yields the useful identity
\begin{equation}\label{eq:beta-plus-one}
  \beta_\ell+1
  =\frac{2P\tau_\ell+2p-n}{p}.
\end{equation}
Since \(\tau_\ell\ge\tau_0=(n-p)/P\), the right-hand side is at least \(n/p>1\).  Thus \(\beta_\ell>0\).
\end{proof}

\subsection{The gauge and compactifying variable}

Fix \(\ell\), abbreviate \(s=s_\ell\), \(\tau=\tau_\ell\), and put
\[
\sigma\coloneq s+\tau,\quad m\coloneq m_\ell=\frac{\sigma}{a}.
\]
Define
\begin{equation}\label{eq:gauge-and-t}
  t\coloneq \frac{r^a}{1+r^a}\in(0,1),
  \quad
  g_\ell(r)\coloneq r^s(1+r^a)^{-m},
  \quad
  f(r)\coloneq g_\ell(r)z(t).
\end{equation}
The two powers in \(g_\ell\) incorporate simultaneously the admissible behaviour \(r^s\) at zero and \(r^{-\tau}\) at infinity.

For clarity, we record the derivative calculation in full.  Set
\begin{equation}\label{eq:A-and-h}
  A(t)\coloneq at(1-t),
  \quad
  h(t)\coloneq s-\sigma t.
\end{equation}
Then
\[
r\frac{\dd t}{\dd r}=A(t),
\quad
r\frac{g_\ell'}{g_\ell}=h(t),
\]
and therefore
\begin{equation}\label{eq:first-derivative}
  rf'=g_\ell\left(hz+Az_t\right).
\end{equation}
Differentiating this identity with respect to \(r\), and using \(r\partial_r=A\partial_t\), gives
\[
r(rf')'
=
g_\ell
\left[
h(hz+Az_t)
+A\partial_t(hz+Az_t)
\right].
\]
Since \(r(rf')'=rf'+r^2f''\), it follows that
\begin{equation}\label{eq:second-derivative}
  r^2f''
  =
  g_\ell
  \left[
  A^2z_{tt}
  +A(2h-1+A_t)z_t
  +\left(h(h-1)-\sigma A\right)z
  \right].
\end{equation}

\begin{proposition}
\label{prop:jacobi-ode}
Under \eqref{eq:gauge-and-t}, equation \eqref{eq:radial-normalized-equation} is equivalent to
\begin{equation}\label{eq:jacobi-equation}
  t(1-t)z_{tt}
  +
  \left[
  \alpha_\ell+1
  -(\alpha_\ell+\beta_\ell+2)t
  \right]z_t
  +
  \left(
  \frac{\mu}{C_{n,p}}-V_\ell
  \right)z
  =0,
\end{equation}
where
\begin{equation}\label{eq:Vell}
  V_\ell
  \coloneq 
  m_\ell\left(
  \frac{P(2s_\ell+n)}{p}
  \right)
  +\frac{n(p-2)P}{p^2}s_\ell
  =
  m_\ell(m_\ell+\delta-1).
\end{equation}
\end{proposition}

\begin{proof}
We retain the abbreviations \(s=s_\ell\), \(\tau=\tau_\ell\), \(\sigma=s+\tau\), and \(m=m_\ell=\sigma/a\), and set
\[
L(t)\coloneq nP-1-n(p-2)t.
\]
Multiplying \eqref{eq:radial-normalized-equation} by \(r^2/g_\ell\), and using
\[
\frac{r^a}{1+r^a}=t,
\quad
\frac{r^a}{(1+r^a)^2}=t(1-t),
\]
gives
\begin{align}
  0=
  P\left[
  A^2z_{tt}
  +A(2h-1+A_t)z_t
  +\left(h(h-1)-\sigma A\right)z
  \right]+L(t)(hz+Az_t)-\lambda_\ell z
  +\widehat\mu t(1-t)z.
  \label{eq:pre-Jacobi-substitution}
\end{align}
We now consider the three coefficients separately.

Since \(A=at(1-t)\), the coefficient of \(z_{tt}\) is
\begin{equation}\label{eq:ztt-coefficient}
  PA^2=Pa^2t^2(1-t)^2.
\end{equation}

The coefficient of \(z_t\) is
\[
A\left[P(2h-1+A_t)+L(t)\right].
\]
Using
\[
h=s-\sigma t,
\quad
A_t=a(1-2t),
\quad
Pa=p,
\]
we find
\begin{align*}
  P(2h-1+A_t)+L(t)
  &=
  P(2s-1+a)+nP-1-
  \left[2P\sigma+2Pa+n(p-2)\right]t\\
  &=P(2s+n)
  -\left[2P\sigma+2p+n(p-2)\right]t.
\end{align*}
Consequently,
\begin{align}
  A\left[P(2h-1+A_t)+L(t)\right]
  =
  Pa^2t(1-t)
  \left[
  \frac{P(2s+n)}{p}
  -\left(
  \frac{2P\sigma}{p}+2+\frac{n(p-2)}{p}
  \right)t
  \right].
  \label{eq:zt-coefficient}
\end{align}
By the definitions of \(\alpha_\ell\), \(\beta_\ell\), \(m\), and \(\delta\),
\[
\frac{P(2s+n)}{p}=\alpha_\ell+1,
\quad
\frac{2P\sigma}{p}+2+\frac{n(p-2)}{p}
=2m+\delta
=\alpha_\ell+\beta_\ell+2.
\]
Thus \eqref{eq:zt-coefficient} is precisely
\[
Pa^2t(1-t)
\left[
\alpha_\ell+1
-(\alpha_\ell+\beta_\ell+2)t
\right].
\]

It remains to simplify the coefficient of \(z\).  Put
\begin{equation}\label{eq:Q-definition}
  Q(t)\coloneq 
  P\left[h(h-1)-\sigma A\right]
  +L(t)h-\lambda_\ell
  +\widehat\mu t(1-t).
\end{equation}
This is a polynomial of degree at most two.  Its constant term is
\begin{align*}
  Q(0)
  =P(s^2-s)+(nP-1)s-\lambda_\ell=Ps^2+Ds-\lambda_\ell=0
\end{align*}
by \eqref{eq:origin-indicial-simplified}.  At \(t=1\), one has \(A(1)=0\), \(h(1)=-\tau\), and \(L(1)=n-1\), whence
\begin{align*}
  Q(1)
  =P\tau(\tau+1)-(n-1)\tau-\lambda_\ell=P\tau^2-(n-p)\tau-\lambda_\ell=0
\end{align*}
by \eqref{eq:infinity-indicial}.  Therefore \(Q(t)=q_1t(1-t)\), where \(q_1\) is the coefficient of \(t\) in \(Q\).  Expanding \eqref{eq:Q-definition} gives
\begin{align*}
  q_1
  &=P\sigma(1-2s-a)-(nP-1)\sigma
    -n(p-2)s+\widehat\mu\\
  &=\widehat\mu-P\sigma(2s+n)-n(p-2)s.
\end{align*}
By \eqref{eq:Vell},
\begin{equation}\label{eq:V-from-z-coefficient}
  Pa^2V_\ell
  =P\sigma(2s+n)+n(p-2)s.
\end{equation}
Then the coefficient of \(z\) is
\begin{equation}\label{eq:z-coefficient}
  Q(t)
  =Pa^2t(1-t)
  \left(
  \frac{\widehat\mu}{Pa^2}-V_\ell
  \right).
\end{equation}
Since \(a=p/P\) and \(C_{n,p}=(p^2/P)B^{p-2}\), we have
\[
\frac{\widehat\mu}{Pa^2}=\frac{\mu}{C_{n,p}}.
\]

For completeness, we verify the second expression for \(V_\ell\). Subtracting \eqref{eq:origin-indicial-simplified} and \eqref{eq:infinity-indicial} gives
\begin{equation}\label{eq:indicial-difference-for-V}
  P(s-\tau)\sigma+Ds+(n-p)\tau=0.
\end{equation}
Equation \eqref{eq:indicial-difference-for-V} implies
\[
P\sigma(2s+n)+n(p-2)s
=P\sigma^2+(nP-n+p)\sigma.
\]
By \eqref{eq:V-from-z-coefficient}, dividing by \(Pa^2=p^2/P\), and using \(am=\sigma\), \(\delta=2+\frac{n(p-2)}{p}\), yields
\[
V_\ell
=m(m+\delta-1).
\]

Finally, substituting \eqref{eq:ztt-coefficient}, \eqref{eq:zt-coefficient}, and \eqref{eq:z-coefficient} into \eqref{eq:pre-Jacobi-substitution} and dividing by
\(Pa^2t(1-t)\) gives \eqref{eq:jacobi-equation}.
\end{proof}

Define the Jacobi differential expression
\begin{equation}\label{eq:Jacobi-expression}
  J_\ell z
  \coloneq 
  -t(1-t)z''
  -
  \left[
  \alpha_\ell+1
  -(\alpha_\ell+\beta_\ell+2)t
  \right]z'.
\end{equation}
Then \eqref{eq:jacobi-equation} becomes
\begin{equation}\label{eq:shifted-Jacobi-formal}
  C_{n,p}(J_\ell+V_\ell)z=\mu z.
\end{equation}
It remains to prove that the transformed closed form is unitarily equivalent to the closed form defining \(T_\ell\), with the endpoint conditions determined by \eqref{eq:radial-cores}.

\section{Form-level Jacobi identification and endpoint closure}
\label{sec:form-identification}

\subsection{The unitary map and the exact transformed core}
Set
\begin{equation}\label{eq:Jacobi-spaces}
  \varpi_\ell(t)\coloneq t^{\alpha_\ell}(1-t)^{\beta_\ell},
  \quad
  \pi_\ell(t)\coloneq t^{\alpha_\ell+1}(1-t)^{\beta_\ell+1},
\end{equation}
and
\[
 \mathscr{J}_\ell
 \coloneq 
 L^2\left((0,1),\varpi_\ell(t)\dd t\right).
\]
The change of variables in \eqref{eq:gauge-and-t} gives the exact norm identity
\begin{equation}\label{eq:norm-transform}
 \int_0^\infty
 \rho(r)|g_\ell(r)z(t(r))|^2r^{n-1}\dd r
 =
 \frac{1}{a}
 \int_0^1|z(t)|^2\varpi_\ell(t)\dd t.
\end{equation}
Thus
\begin{equation}\label{eq:unitary-U}
 (\mathcal{U}_\ell z)(r)
 \coloneq 
 \sqrt{a}\,g_\ell(r)z(t(r))=\sqrt{a}f(r)
\end{equation}
defines a unitary map from \(\mathscr{J}_\ell\) onto \(\mathscr{K}_v\).

The preimage under \(\mathcal{U}_\ell\) of the radial core \eqref{eq:radial-cores} is
\begin{equation}\label{eq:transformed-cores}
  \mathcal{G}_\ell
  \coloneq 
  \begin{cases}
    \left\{
    z\in C^\infty([0,1)):
    \begin{array}{l}
    z(t)=c(1-t)^{-m_0}\text{ near }t=0,\\
    z(t)=0\text{ near }t=1
    \end{array}
    \right\},
    &\ell=0,\\[5mm]
    C_c^\infty(0,1),&\ell\geq 1.
  \end{cases}
\end{equation}
For \(\ell=0\), \(g_0=(1-t)^{m_0}\); hence \(\mathcal{U}_0 z\) is constant near \(r=0\) if and only if \(z(t)=c(1-t)^{-m_0}\) near \(t=0\).  For \(\ell\geq 1\), the map
\(r\mapsto t\) is a smooth diffeomorphism from \((0,\infty)\) onto \((0,1)\), and \(g_\ell\) is smooth and positive there.  Therefore
\[
\mathcal{U}_\ell^{-1}C_c^\infty(0,\infty)=C_c^\infty(0,1).
\]
The factor \(\sqrt{a}\) is absorbed into the constant \(c\).

On \(\mathcal{G}_\ell\) define
\begin{equation}\label{eq:Jacobi-form}
  \mathfrak{j}_\ell[z,y]
  \coloneq 
  \int_0^1
  z'(t)y'(t)\pi_\ell(t)\dd t.
\end{equation}

\begin{proposition}
\label{prop:exact-form-identity}
For \(z,y\in\mathcal{G}_\ell\),
\begin{equation}\label{eq:exact-form-identity}
  \mathfrak{b}_\ell[\mathcal{U}_\ell z,\mathcal{U}_\ell y]
  =
  C_{n,p}
  \left(
  \mathfrak{j}_\ell[z,y]
  +V_\ell\ip{z}{y}_{\mathscr{J}_\ell}
  \right).
\end{equation}
Consequently, the closure from \(\mathcal{D}_\ell^{\mathrm{rad}}\) of \(\mathfrak{b}_\ell\) is unitarily equivalent to the closure from \(\mathcal{G}_\ell\) of \(C_{n,p}(\mathfrak{j}_\ell+V_\ell)\).
\end{proposition}

\begin{proof}
Write \(G(r)\coloneq g_\ell(r)\), and set
\[
f_z\coloneq \mathcal{U}_\ell z=\sqrt{a}\,G(r)z(t),
\quad
f_y\coloneq \mathcal{U}_\ell y=\sqrt{a}\,G(r)y(t).
\]
By \eqref{eq:first-derivative},
\begin{equation}\label{eq:form-transformed-derivatives}
  r f_z'
  =\sqrt{a}\,G(hz+Az'),
  \quad
  r f_y'
  =\sqrt{a}\,G(hy+Ay'),
\end{equation}
where here and below primes on \(z\) and \(y\) denote differentiation with respect to \(t\). 

We first record the common weight produced by this substitution.  A direct use of \eqref{eq:w-weight}, \(G=r^s(1+r^a)^{-m}\), and \(r^a=t/(1-t)\) gives
\begin{align}
  w(r)G(r)^2r^{n-2}
  &=
  B^{p-2}
  r^{(p-2)/P+2s+n-2}
  (1+r^a)^{-n(p-2)/p-2m}\notag\\
  &=B^{p-2}t^{\alpha_\ell}(1-t)^{\beta_\ell}
  =B^{p-2}\varpi_\ell(t).
  \label{eq:energy-weight-transform}
\end{align}
Indeed, the exponent of \(t\) is
\[
\frac{1}{a}\left(\frac{p-2}{P}+2s+n-2\right)
=\frac{P(2s+n)}{p}-1
=\alpha_\ell,
\]
and the exponent of \(1-t\) is
\[
\frac{n(p-2)}{p}+2m-\alpha_\ell=\beta_\ell.
\]

Substituting \eqref{eq:form-transformed-derivatives} into \eqref{eq:radial-form}, changing variables from \(r\) to \(t\), and using \eqref{eq:energy-weight-transform}, we obtain
\begin{align}
  \mathfrak{b}_\ell[f_z,f_y]
  =&C_{n,p}\int_0^1
  \biggl\{
  \pi_\ell z'y'
  +\frac{h}{a}\varpi_\ell
        (z y)'
  +\frac{\varpi_\ell}{aA}
    \left(h^2+\frac{\lambda_\ell}{P}\right)
    z y
  \biggr\}\dd t.
  \label{eq:form-before-mixed-integration}
\end{align}
Here we used
\[
 C_{n,p}=Pa^2B^{p-2},
 \quad
 \pi_\ell=t(1-t)\varpi_\ell=\frac{A}{a}\varpi_\ell.
\]

We now integrate the mixed term in \eqref{eq:form-before-mixed-integration}.  Its boundary contribution is
\begin{equation}\label{eq:mixed-boundary-term}
  \left[
  \frac{h(t)}{a}\varpi_\ell(t)z(t)y(t)
  \right]_{t=0}^{t=1}.
\end{equation}
It vanishes at \(t=1\), since every element of \(\mathcal{G}_\ell\) vanishes near \(1\).  If \(\ell\geq 1\), it also vanishes at \(t=0\) because \(\mathcal{G}_\ell=C_c^\infty(0,1)\).  If \(\ell=0\), then \(s_0=0\), so \(h(t)=-\sigma t\) near \(0\), while \(z\) and \(y\) are bounded there.  Hence the expression in
\eqref{eq:mixed-boundary-term} is \(O(t^{\alpha_0+1})\), which tends to zero because \(\alpha_0>-1\) by \Cref{lem:parameter-bounds}.

After this integration by parts, the coefficient of \(\varpi_\ell zy\) is
\begin{equation}\label{eq:transformed-zero-order-W}
  W(t)
  \coloneq 
  \frac{h^2+\lambda_\ell/P}{aA}
  -\frac{1}{a}
  \left(
  h'+h\frac{\varpi_\ell'}{\varpi_\ell}
  \right).
\end{equation}
We claim that \(W(t)\equiv V_\ell\).  Since
\[
h'=-\sigma,
\quad
\frac{\varpi_\ell'}{\varpi_\ell}
=\frac{\alpha_\ell}{t}
-\frac{\beta_\ell}{1-t},
\]
multiplication of \eqref{eq:transformed-zero-order-W} by \(aA\) gives the quadratic polynomial
\begin{equation}\label{eq:N-form-potential}
  N(t)
  \coloneq 
  h^2+\frac{\lambda_\ell}{P}
  +a\sigma t(1-t)
  -ah\left[
  \alpha_\ell(1-t)-\beta_\ell t
  \right].
\end{equation}
At the left endpoint,
\[
N(0)=s^2+\frac{\lambda_\ell}{P}-as\alpha_\ell=0.
\]
Indeed, \(a\alpha_\ell=2s+n-a\), and the last equality is precisely \(Ps^2+Ds-\lambda_\ell=0\), since \(D/P=n-a\). At the right endpoint,
\[
N(1)=\tau^2+\frac{\lambda_\ell}{P}-a\tau\beta_\ell=0.
\]
Here
\[
a\beta_\ell=2\tau-\frac{n-p}{P},
\]
so the last equality is exactly \(P\tau^2-(n-p)\tau-\lambda_\ell=0\). Thus \(N\) is a quadratic polynomial vanishing at both \(0\) and \(1\).

The coefficient of \(t\) in \eqref{eq:N-form-potential} is
\begin{align*}
  -2s\sigma+a\sigma
  +a\left[s(\alpha_\ell+\beta_\ell)
            +\sigma\alpha_\ell\right]=
  \sigma(2s+n)+\frac{n(p-2)}{P}s
  =a^2V_\ell.
\end{align*}
For the first equality we used
\[
\alpha_\ell+\beta_\ell=2m+\frac{n(p-2)}{p},
\quad
\alpha_\ell=\frac{P(2s+n)}{p}-1,
\quad
m=\frac{\sigma}{a};
\]
the last equality follows from \eqref{eq:Vell}.  Consequently,
\[
N(t)=a^2V_\ell t(1-t),
\]
and hence \(W(t)=V_\ell\), as claimed.

Returning to \eqref{eq:form-before-mixed-integration}, we have proved
\[
\mathfrak{b}_\ell[f_z,f_y]
=
C_{n,p}
\int_0^1
\left(
z'y'\,\pi_\ell
+V_\ell zy\,\varpi_\ell
\right)\dd t,
\]
which is \eqref{eq:exact-form-identity}.

Finally, \(\mathcal{U}_\ell\) is unitary from \(\mathscr{J}_\ell\) onto \(\mathscr{K}_v\), and
\[
\mathcal{U}_\ell\mathcal G_\ell=\mathcal D_\ell^{\mathrm{rad}}.
\]
By \eqref{eq:norm-transform} and \eqref{eq:exact-form-identity}, for every \(z\in\mathcal G_\ell\),
\[
\|\mathcal{U}_\ell z\|_{\mathfrak{b}_\ell}^{2}
=
\|z\|_{\mathscr{J}_\ell}^{2}
+
C_{n,p}
\left(
\mathfrak{j}_\ell[z,z]
+
V_\ell\|z\|_{\mathscr{J}_\ell}^{2}
\right).
\]
Consequently, a sequence in \(\mathcal G_\ell\) is Cauchy with respect to the norm on the right-hand side if and only if its image under \(\mathcal{U}_\ell\) is Cauchy with respect to \(\|\cdot\|_{\mathfrak{b}_\ell}\).  Hence \(\mathcal{U}_\ell\) extends to the corresponding completions, and \eqref{eq:exact-form-identity} extends to the closed forms.
\end{proof}

\subsection{Why Jacobi polynomials satisfy the selected endpoint conditions}

The formal solutions of \eqref{eq:shifted-Jacobi-formal} that are
polynomials are
\begin{equation}\label{eq:shifted-Jacobi-polynomials}
 Z_{\ell,k}(t)
 \coloneq 
 P_k^{(\alpha_\ell,\beta_\ell)}(1-2t),
 \quad k\in\mathbb{N}_0,
\end{equation}
and they satisfy
\begin{equation}\label{eq:Jacobi-polynomial-eigenvalue}
 J_\ell Z_{\ell,k}
 =
 k(k+\alpha_\ell+\beta_\ell+1)Z_{\ell,k}.
\end{equation}

\begin{lemma}
\label{lem:polynomial-form-domain}
Every polynomial \(Z\) on \([0,1]\) belongs to the completion of
\(\mathcal{G}_\ell\) in the norm
\begin{equation}\label{eq:Jacobi-form-norm}
 \|z\|_{\mathscr{J}_\ell}^2+\mathfrak{j}_\ell[z,z].
\end{equation}
\end{lemma}

\begin{proof}
Choose \(\chi\in C^\infty(\mathbb{R})\) such that
\[
 0\leq\chi\leq 1,\quad
 \chi=0\ \text{on }(-\infty,1],\quad
 \chi=1\ \text{on }[2,\infty),
\]
and set
\[
 \chi_{0,\varepsilon}(t)
 \coloneq 
 \chi\left(\frac{t}{\varepsilon}\right),
 \quad
 \chi_{1,\varepsilon}(t)
 \coloneq 
 \chi\left(\frac{1-t}{\varepsilon}\right).
\]
Then
\[
 |\chi_{0,\varepsilon}'|
 +
 |\chi_{1,\varepsilon}'|
 \leq C\varepsilon^{-1}.
\]

Suppose first that \(\ell\geq 1\), and define
\[
 Z_\varepsilon
 \coloneq 
 \chi_{0,\varepsilon}\chi_{1,\varepsilon}Z
 \in C_c^\infty(0,1)=\mathcal{G}_\ell.
\]
Since \(Z,Z'\in L^\infty(0,1)\), at \(t=0\) we have
\begin{align*}
 \int_0^{2\varepsilon}
 |Z_\varepsilon-Z|^2\varpi_\ell\,\dd t
 &\leq
 C\int_0^{2\varepsilon}t^{\alpha_\ell}\dd t
 =
 O(\varepsilon^{\alpha_\ell+1}),\\
 \int_0^{2\varepsilon}
 |(Z_\varepsilon-Z)'|^2\pi_\ell\,\dd t
 &\leq
 C\left(\int_0^{2\varepsilon}t^{\alpha_\ell+1}\dd t
 +
 \frac{1}{\varepsilon^2}
 \int_\varepsilon^{2\varepsilon}
 t^{\alpha_\ell+1}\dd t\right)\\
 &=
 O(\varepsilon^{\alpha_\ell}).
\end{align*}
Similarly, at \(t=1\),
\begin{align*}
 \int_{1-2\varepsilon}^1
 |Z_\varepsilon-Z|^2\varpi_\ell\,\dd t
 &=
 O(\varepsilon^{\beta_\ell+1}),\\
 \int_{1-2\varepsilon}^1
 |(Z_\varepsilon-Z)'|^2\pi_\ell\,\dd t
 &=
 O(\varepsilon^{\beta_\ell}).
\end{align*}
Therefore,
\[
 \|Z_\varepsilon-Z\|_{\mathscr{J}_\ell}^2
 +
 \mathfrak{j}_\ell[Z_\varepsilon-Z,Z_\varepsilon-Z]
 \leq
 C\left(
 \varepsilon^{\alpha_\ell}
 +
 \varepsilon^{\beta_\ell}
 \right)
 \to 0,
\]
because \(\alpha_\ell,\beta_\ell>0\) by \Cref{lem:parameter-bounds}.

Now let \(\ell=0\).  Set
\[
 Z_*(t)\coloneq Z(0)(1-t)^{-m_0}.
\]
Since \(Z_*(0)=Z(0)\),
\[
 Z_*-Z=O(t),
 \quad
 Z_*'-Z'=O(1)
 \quad (t\downarrow0).
\]
Define
\[
 \widetilde Z_\varepsilon
 \coloneq 
 Z+
 (1-\chi_{0,\varepsilon})(Z_*-Z).
\]
Then
\[
 \widetilde Z_\varepsilon=Z_*
 \quad\text{near }0,
 \quad
 \widetilde Z_\varepsilon=Z
 \quad\text{for }t\geq 2\varepsilon.
\]
Moreover,
\begin{align*}
 \|\widetilde Z_\varepsilon-Z\|_{\mathscr{J}_0}^2
 &\leq
 C\int_0^{2\varepsilon}t^{\alpha_0+2}\dd t
 =
 O(\varepsilon^{\alpha_0+3}),\\
 \mathfrak{j}_0[
 \widetilde Z_\varepsilon-Z,
 \widetilde Z_\varepsilon-Z]
 &\leq
 C\int_0^{2\varepsilon}t^{\alpha_0+1}\dd t
 =
 O(\varepsilon^{\alpha_0+2}).
\end{align*}

Finally, put
\[
Z_\varepsilon
\coloneq 
\chi_{1,\varepsilon}\widetilde Z_\varepsilon.
\]
Then \(Z_\varepsilon\in\mathcal{G}_0\), and the right-endpoint estimates give
\[
\|Z_\varepsilon-Z\|_{\mathscr{J}_0}^2
+
\mathfrak{j}_0[Z_\varepsilon-Z,Z_\varepsilon-Z]
\leq
C\left(
\varepsilon^{\alpha_0+2}
+
\varepsilon^{\beta_0}
\right)
\to 0,
\]
because \(\alpha_0>-1\) and \(\beta_0>0\) by \Cref{lem:parameter-bounds}.  Hence \(Z\) belongs to the form closure of \(\mathcal{G}_\ell\) for every \(\ell\geq 0\).
\end{proof}

\begin{proposition}
\label{prop:operator-domain-eigenfunctions}
For every \(\ell,k\geq 0\), \(Z_{\ell,k}\) belongs to the closed
transformed form domain and
\begin{equation}\label{eq:weak-Jacobi-eigenidentity}
  \mathfrak{j}_\ell[Z_{\ell,k},y]
  =
  k(k+\alpha_\ell+\beta_\ell+1)
  \ip{Z_{\ell,k}}{y}_{\mathscr{J}_\ell}
\end{equation}
for every \(y\) in that form domain.  Consequently, \(\mathcal{U}_\ell Z_{\ell,k}\in\cD(T_\ell)\) and
\begin{equation}\label{eq:T-eigenvalue-preliminary}
  T_\ell\mathcal{U}_\ell Z_{\ell,k}
  =
  C_{n,p}
  \left[
  V_\ell+k(k+\alpha_\ell+\beta_\ell+1)
  \right]\mathcal{U}_\ell Z_{\ell,k}.
\end{equation}
\end{proposition}

\begin{proof}
By \cref{lem:polynomial-form-domain}, \(Z_{\ell,k}\) belongs to the domain of the closed transformed form.  For \(y\in\mathcal{G}_\ell\), integration by parts gives
\[
\mathfrak{j}_\ell[Z_{\ell,k},y]
=
-\int_0^1
(\pi_\ell Z_{\ell,k}')'y\,\dd t
+
\left[
\pi_\ell Z_{\ell,k}'y
\right]_{t=0}^{t=1}.
\]
The boundary term at \(t=1\) is zero because \(y\) vanishes there.  At \(t=0\), it is also zero, since
\[
\pi_\ell(t)=O(t^{\alpha_\ell+1}),
\quad
\alpha_\ell+1>0,
\]
and \(Z_{\ell,k}'\) and \(y\) are bounded near \(t=0\).  Therefore, using
\[
J_\ell Z_{\ell,k}
=
-\frac{1}{\varpi_\ell}
(\pi_\ell Z_{\ell,k}')'
=
k(k+\alpha_\ell+\beta_\ell+1)Z_{\ell,k},
\]
we obtain
\[
\mathfrak{j}_\ell[Z_{\ell,k},y]
=
k(k+\alpha_\ell+\beta_\ell+1)
\ip{Z_{\ell,k}}{y}_{\mathscr{J}_\ell}
\]
for every \(y\in\mathcal{G}_\ell\).

The space \(\mathcal{G}_\ell\) is dense in \(\mathcal{U}_\ell^{-1}\mathcal{D}_\ell(\mathfrak{b}_\ell)\) with respect to the norm \(\|\cdot\|_{\mathscr{J}_\ell}^2+\mathfrak{j}_\ell [\cdot,\cdot]\), or equivalently,
\[
y\longmapsto\|\mathcal{U}_\ell y\|_{\mathfrak{b}_\ell}.
\]
For fixed \(Z_{\ell,k}\), both sides of the preceding identity are continuous linear functionals of \(y\) in this norm.  Therefore \eqref{eq:weak-Jacobi-eigenidentity} holds for every
\[
y\in\mathcal{U}_\ell^{-1}\mathcal D_\ell(\mathfrak b_\ell).
\]

Applying \cref{prop:exact-form-identity}, we consequently obtain
\[
\mathfrak{b}_\ell
[\mathcal{U}_\ell Z_{\ell,k},\mathcal{U}_\ell y]
=
C_{n,p}
\left[
V_\ell+k(k+\alpha_\ell+\beta_\ell+1)
\right]
\ip{\mathcal{U}_\ell Z_{\ell,k}}{\mathcal{U}_\ell y}_{\mathscr{K}_v}
\]
for every \(y\) in the domain of the closed transformed form.  Since \(\mathcal{U}_\ell\) maps this domain onto \(\mathcal{D}_\ell(\mathfrak{b}_\ell)\), the representation theorem implies that
\[
\mathcal{U}_\ell Z_{\ell,k}\in\mathcal{D}(T_\ell)
\]
and gives \eqref{eq:T-eigenvalue-preliminary}.
\end{proof}

\section{The complete two-index spectrum}
\label{sec:complete-spectrum}

\subsection{Eigenvalues and eigenfunctions}

By \eqref{eq:jacobi-parameters-intro},
\[
\alpha_\ell+\beta_\ell+1=\delta+2m_\ell-1.
\]
Together with \eqref{eq:Vell}, this reduces the eigenvalue in \eqref{eq:T-eigenvalue-preliminary} to
\begin{align}
  \mu_{\ell,k}
  &=
  C_{n,p}
  \left[
  m_\ell(m_\ell+\delta-1)
  +k(k+\delta+2m_\ell-1)
  \right]\notag\\
  &=
  C_{n,p}(m_\ell+k)(m_\ell+k+\delta-1).
  \label{eq:muellk}
\end{align}
By the definition of \(\cU_\ell\), the radial eigenfunction corresponding to \(Z_{\ell,k}\) is, up to the constant factor \(\sqrt{a}\),
\begin{equation}\label{eq:fellk}
  f_{\ell,k}(r)
  =
  r^{s_\ell}(1+r^a)^{-m_\ell}
  P_k^{(\alpha_\ell,\beta_\ell)}
  \left(\frac{1-r^a}{1+r^a}\right).
\end{equation}

\subsection{Completeness and absence of additional spectrum}

\begin{proposition}
\label{prop:radial-completeness}
For each \(\ell\geq 0\), the functions \(\{\mathcal{U}_\ell Z_{\ell,k}\}_{k\geq 0}\) form a complete orthogonal eigenbasis of \(T_\ell\) in \(\mathscr{K}_v\).  In particular,
\[
\spec(T_\ell)=\{\mu_{\ell,k}:k\in\mathbb{N}_0\},
\]
with every eigenvalue in this fixed angular block simple.
\end{proposition}

\begin{proof}
By \cref{lem:parameter-bounds}, the Jacobi weight \(\varpi_\ell=t^{\alpha_\ell}(1-t)^{\beta_\ell}\) is integrable on \((0,1)\).  The shifted Jacobi polynomials are mutually orthogonal in \(\mathscr{J}_\ell\).  Their linear span is the set of all polynomials. Continuous functions are dense in \(\mathscr{J}_\ell\), and the Weierstrass theorem makes polynomials uniformly dense in continuous functions; hence the Jacobi polynomials are complete in \(\mathscr{J}_\ell\).

\Cref{prop:operator-domain-eigenfunctions} shows that every element of this complete orthogonal family is an eigenfunction of the self-adjoint transformed operator.  A self-adjoint operator possessing a complete orthogonal eigenbasis is the diagonal operator determined by those eigenpairs.  Since \(\mathcal{U}_\ell\) is unitary, the same conclusion holds for \(T_\ell\).
\end{proof}

\section{Proof of Theorem \ref{thm:main-spectrum} and Corollary \ref{cor:E3-intro}}
\label{sec:E3}

\begin{proof}[Proof of \Cref{thm:main-spectrum}]
By \Cref{prop:angular-decomposition},
\begin{equation}
  \mathscr{H}_v
  \simeq
  \widehat{\bigoplus}_{\ell=0}^{\infty}
  \left(\mathscr{K}_v\otimes\cH_\ell\right),
  \quad
  \mathcal{L}_v
  \simeq
  \widehat{\bigoplus}_{\ell=0}^{\infty}
  \left(T_\ell\otimes I_{\cH_\ell}\right).
  \label{eq:main-proof-direct-sum}
\end{equation}
On the other hand, \cref{prop:radial-completeness} gives
\begin{equation}
  \spec(T_\ell)
  =
  \{\mu_{\ell,k}:k\in\mathbb{N}_0\},
  \quad
  \overline{\spanop\{
  \mathcal{U}_\ell Z_{\ell,k}:k\in\mathbb{N}_0
  \}}
  =
  \mathscr{K}_v,
  \label{eq:main-proof-radial-spectrum}
\end{equation}
where
\[
\mu_{\ell,k}
=
C_{n,p}(m_\ell+k)(m_\ell+k+\delta-1).
\]
Since
\[
Z_{\ell,k}(t)
=
P_k^{(\alpha_\ell,\beta_\ell)}(1-2t),
\quad
t=\frac{r^a}{1+r^a},
\]
we have
\[
1-2t
=
\frac{1-r^a}{1+r^a},
\]
and, up to the harmless normalization factor \(\sqrt{a}\),
\begin{align*}
  (\mathcal{U}_\ell Z_{\ell,k})(r)
  &=
  r^{s_\ell}(1+r^a)^{-m_\ell}
  P_k^{(\alpha_\ell,\beta_\ell)}
  \left(\frac{1-r^a}{1+r^a}\right).
\end{align*}
Combining this identity with
\eqref{eq:main-proof-direct-sum} and
\eqref{eq:main-proof-radial-spectrum} shows that the functions
\[
\Phi_{\ell,k,j}(r,\theta)
=
r^{s_\ell}(1+r^a)^{-m_\ell}
P_k^{(\alpha_\ell,\beta_\ell)}
\left(
\frac{1-r^a}{1+r^a}
\right)
Y_{\ell,j}(\theta),
\]
where \(\ell,k\in\mathbb{N}_0,\,1\leq j\leq d_{n,\ell}\) form a complete orthogonal family in \(\mathscr{H}_v\).  Their
eigenvalues are
\[
\mu_{\ell,k}
=
C_{n,p}(m_\ell+k)(m_\ell+k+\delta-1).
\]
For each fixed pair \((\ell,k)\), the span of
\[
\left\{
\Phi_{\ell,k,j}:
1\leq j\leq d_{n,\ell}
\right\}
\]
is precisely the subspace \(E_{\ell,k}\) defined in
\eqref{eq:separated-eigenspace}.  In particular,
\[
\dim E_{\ell,k}=d_{n,\ell}.
\]
For each fixed \((\ell,k)\), the separated eigenspace is isomorphic to \(\cH_\ell\), and therefore has dimension \(d_{n,\ell}\).

It remains to determine when two pairs give the same numerical eigenvalue.  Define
\begin{equation}
\label{eq:Lambda}
  \Lambda(\xi)
  \coloneq 
  C_{n,p}\xi(\xi+\delta-1).
\end{equation}
Since
\[
\lambda_\ell=\ell(\ell+n-2)
\]
is strictly increasing, the admissible roots \(s_\ell\) and \(\tau_\ell\), and hence
\[
m_\ell=\frac{p-1}{p}(s_\ell+\tau_\ell),
\]
are strictly increasing in \(\ell\).  Moreover,
\[
m_\ell\geq m_0=\frac{n-p}{p},
\]
and, for \(\xi\geq m_0\),
\begin{align}
  \Lambda'(\xi)
  =
  C_{n,p}(2\xi+\delta-1)\geq
  C_{n,p}(2m_0+\delta-1)
  =
  C_{n,p}(n-1)
  >0.
  \label{eq:Lambda-derivative-at-m0}
\end{align}
Thus \(\Lambda\) is strictly increasing on \([m_0,\infty)\).  Since
\[
\mu_{\ell,k}=\Lambda(m_\ell+k),
\]
it follows that
\begin{align*}
  \mu_{\ell,k}=\mu_{\ell',k'}
  &\iff
  \Lambda(m_\ell+k)=\Lambda(m_{\ell'}+k')\\
  &\iff
  m_\ell+k=m_{\ell'}+k',
\end{align*}
which proves \eqref{eq:crossing-rule-intro}.

Finally,
\[
m_\ell\to \infty
\quad(\ell\to\infty),
\]
and
\[
\Lambda(\xi)\to \infty
\quad(\xi\to\infty).
\]
Hence, for every \(M>0\),
\[
\#\left\{
(\ell,k)\in\mathbb{N}_0^2:\mu_{\ell,k}\leq M
\right\}
<\infty.
\]
Therefore the spectrum is discrete and every eigenvalue has finite
multiplicity.  More precisely, for every
\(\mu\in\spec(\mathcal{L}_v)\),
\[
\ker(\mathcal{L}_v-\mu I)
=
\bigoplus_{\substack{\ell,k\in\mathbb{N}_0\\
                     \mu_{\ell,k}=\mu}}
E_{\ell,k},
\]
where the sum is orthogonal.  Consequently,
\[
\dim\ker(\mathcal{L}_v-\mu I)
=
\sum_{\substack{\ell,k\in\mathbb{N}_0\\
                 \mu_{\ell,k}=\mu}}
d_{n,\ell}.
\]
This proves all assertions.
\end{proof}

We now assume \(1<p<2\).  First observe that the \(\ell=1\) indicial roots are explicit:
\begin{equation}\label{eq:l1-roots}
  s_1=\frac{1}{P},
  \quad
  \tau_1=\frac{n-1}{P}.
\end{equation}
Indeed, these values solve \eqref{eq:origin-indicial-simplified} and \eqref{eq:infinity-indicial} with \(\lambda_1=n-1\). Consequently,
\begin{equation}\label{eq:m0m1}
  m_0=\frac{n-p}{p},
  \quad
  m_1=\frac{n}{p}=m_0+1.
\end{equation}
Thus the first spectral index is \(m_0\), represented by pair \((0,0)\), and the next one is \(m_0+1=m_1\), represented by the pairs \((0,1)\) and \((1,0)\).

Because \(\ell\mapsto m_\ell\) is strictly increasing, the third distinct spectral index is
\[
\min\{m_2,m_0+2(=m_1+1)\}.
\]

\begin{proposition}
\label{prop:m2-strict}
If \(n\geq 3\) and \(1<p<2\), then
\begin{equation}\label{eq:m2-inequality}
  m_1<m_2<m_0+2.
\end{equation}
If \(p=2\) and \(n\geq 3\), then \(m_2=m_0+2\).
\end{proposition}

\begin{proof}
Since
\[
\lambda_2=2n>\lambda_1=n-1
\]
and both admissible indicial roots are strictly increasing in \(\lambda_\ell\), we have
\[
m_1<m_2.
\]

Set
\[
\Delta_s\coloneq \sqrt{D^2+8nP},
\quad
\Delta_\tau\coloneq \sqrt{(n-p)^2+8nP}.
\]
From \eqref{eq:sell}--\eqref{eq:mell},
\begin{equation}
\label{eq:m2-difference}
  m_2-(m_0+2)
  =
  \frac{\Delta_s+\Delta_\tau-p(n+2)}{2p}.
\end{equation}
Thus it remains to prove
\[
\Delta_s+\Delta_\tau<p(n+2).
\]

Since \(p=P+1\), direct expansion gives
\begin{align*}
  \Delta_s^2
  &=
  \left((n-1)P-1\right)^2+8nP\\
  &=
  (n-1)^2P^2+(6n+2)P+1,
  \\
  \Delta_\tau^2
  &=
  (n-1-P)^2+8nP\\
  &=
  P^2+(6n+2)P+(n-1)^2.
\end{align*}
Consequently,
\begin{align}
  &p^2(n+2)^2-\Delta_s^2-\Delta_\tau^2
  \notag\\
  &\quad=
  (6n+2)P^2
  +
  2\left(n^2-2n+2\right)P+6n+2
  >0.
  \label{eq:first-squared-difference}
\end{align}

All quantities involved are positive.  Hence
\begin{align*}
  \Delta_s+\Delta_\tau<p(n+2)
  &\iff
  (\Delta_s+\Delta_\tau)^2<p^2(n+2)^2\\
  &\iff
  2\Delta_s\Delta_\tau
  <
  p^2(n+2)^2-\Delta_s^2-\Delta_\tau^2\\
  &\iff
  \begin{aligned}[t]
    4\Delta_s^2\Delta_\tau^2
    <\left[(6n+2)P^2+2\left(n^2-2n+2\right)P+6n+2
      \right]^2.
  \end{aligned}
\end{align*}
A direct calculation gives
\begin{align*}
  \left[
  (6n+2)P^2
  +
  2\left(n^2-2n+2\right)P+6n+2
  \right]^2
  -
  4\Delta_s^2\Delta_\tau^2
  =
  32n(n+1)(1-P^2)^2.
\end{align*}
For \(1<p<2\), one has \(0<P<1\), and therefore
\[
32n(n+1)(1-P^2)^2>0.
\]
The preceding equivalences yield
\[
\Delta_s+\Delta_\tau<p(n+2).
\]
By \eqref{eq:m2-difference},
\[
m_2-(m_0+2)<0,
\]
which proves
\[
m_1<m_2<m_0+2.
\]

If \(p=2\), then \(P=1\), and the above formulas give
\[
\Delta_s^2=\Delta_\tau^2=(n+2)^2.
\]
Thus
\[
\Delta_s+\Delta_\tau=2(n+2)=p(n+2),
\]
and \eqref{eq:m2-difference} yields
\begin{equation}\label{eq:m2m0}
  m_2=m_0+2.
\end{equation}
\end{proof}

\begin{proof}[Proof of \cref{cor:E3-intro}]
The strict monotonicity of \(\Lambda\) from \eqref{eq:Lambda-derivative-at-m0}, the list of the first three spectral indices, and \cref{prop:m2-strict} show that the third distinct eigenvalue is represented only by \((\ell,k)=(2,0)\) when \(1<p<2\). Since \(P_0^{(\alpha_2,\beta_2)}=1\), formulas \eqref{eq:muellk} and \eqref{eq:fellk} give \eqref{eq:mu3-intro} and \eqref{eq:E3-intro}.  Finally,
\[
d_{n,2}
=
\binom{n+1}{2}-1
=
\frac{(n-1)(n+2)}{2}.
\]
\end{proof}

\begin{remark}
For \(p=2\), \eqref{eq:m2m0} shows that the third spectral index \(m_0+2\) is represented by \((2,0)\), \((1,1)\), and \((0,2)\).  Since
\[
\dim\mathcal{H}_2=\frac{(n-1)(n+2)}{2},
\quad
\dim\mathcal{H}_1=n,
\quad
\dim\mathcal{H}_0=1,
\]
the corresponding eigenspace has dimension
\[
\frac{(n-1)(n+2)}{2}+n+1
=
\frac{n(n+3)}{2}.
\]  
\end{remark}

\end{document}